\documentclass[11pt]{article}

\usepackage{amssymb, amsthm, amsmath, mathrsfs, fullpage}
\usepackage{tikz}
\usetikzlibrary{matrix}

\newcommand{\al}[2]{c_{#1}(#2)}

\newcommand{\bigL}[2]{\mathcal{L}(#1,#2)}
\newcommand{\binomial}[2]{{{#1}\choose{#2}}}
\newcommand{\C}{\mathcal{C}}
\newcommand{\CC}{\mathbb{C}}

\newcommand{\pihat}{\hat{\pi}}

\renewcommand{\S}{\mathcal{S}}

\DeclareMathOperator{\des}{des}
\DeclareMathOperator{\Des}{Des}

\DeclareMathOperator{\PPat}{\Pi}
\DeclareMathOperator{\SYT}{SYT}

\theoremstyle{plain}
\newtheorem{theorem}{Theorem}[section]
\newtheorem{lemma}[theorem]{Lemma}
\newtheorem{proposition}[theorem]{Proposition}
\newtheorem{corollary}[theorem]{Corollary}

\theoremstyle{definition}

\newtheorem{remark}[theorem]{Remark}
\newtheorem{sample result}[theorem]{Sample Result}

\title{Descents of $\lambda$-unimodal cycles in a character formula}
\author{Kassie Archer}
\date{}

\usepackage{setspace}
\begin{document}
\maketitle

\begin{abstract}
	We prove an identity conjectured by Adin and Roichman involving the descent set of $\lambda$-unimodal 
	cyclic permutations. These permutations appear in formulas for characters of certain representations of the 
	symmetric group. Such formulas have previously been proven algebraically. In this paper, we present a 
	combinatorial proof for one such formula and discuss the consequences for the distribution of the descent
	set on cyclic permutations.
\end{abstract}

\section{Introduction}

Given a composition $\lambda = (\lambda_1, \lambda_2, \ldots, \lambda_k)$ of $n$, we say a permutation
is \textit{$\lambda$-unimodal} if the permutation, written in its one-line notation, is the concatenation of unimodal
 segments of length $\lambda_i$. (See Figure \ref{Fig:1} for an example.) 
   These
permutations and their descent sets appear in the formulas for certain characters of representations of the 
symmetric group \cite{char, arc, note}. These formulas are of the same form found in Theorem \ref{thm:main}, 
where the sum extends over $\lambda$-unimodal permutations with some extra property which varies based 
on the character. Known formulas for characters include sums over $\lambda$-unimodal permutations which 
are involutions, are in some given Knuth class, or have a given Coxeter length \cite{char}. 

We prove a formula of this type originally conjectured by Ron Adin and Yuval Roichman \cite{yuv}. Suppose 
$\chi$ is the character of the representation of $\S_n$ induced from a primitive linear representation on a 
cyclic subgroup generated by an $n$-cycle. Let $\chi_\lambda$ be its value on the conjugacy class of type 
$\lambda$. Denote by $S(\lambda)$ the set of partial sums 
$\{\lambda_1, \lambda_1+\lambda_2, \ldots, \lambda_1 + \cdots + \lambda_k\}$ and by $\C_\lambda$ the set 
of $\lambda$-unimodal permutations which are also \textit{cyclic}, meaning they can be written in cycle-notation as a single $n$-cycle. 
Theorem \ref{thm:main}, which we will prove in Section \ref{sec:pf}, 
is the main result of this paper.

\begin{theorem}\label{thm:main} 
	For every composition $\lambda$,
	\begin{equation}
	\label{eq:main}
		\chi_\lambda = \sum_{\pi \in \C_\lambda} (-1)^{|\Des(\pi) \setminus S(\lambda)|}.
	\end{equation}
\end{theorem}

It is an simple exercise to show (see Proposition \ref{prop:ind}) that the character described above takes the 
following values. 
\begin{equation}
\label{eq:chi}
	\chi_\lambda = \begin{cases} (k-1)! d^{k-1} \mu(d) & \text{if } \lambda = (d^k) \\ 0 & \text{otherwise} \end{cases}
\end{equation}
where $\mu$ denotes the number-theoretic M\"{o}bius function, which takes values $\mu(d) = (-1)^t$ when $d$ is the square-free product of $t$ distinct primes, and $\mu(d) = 0$ otherwise. 

Previously, these types of character formulas have been proven 
algebraically \cite{Roi1, Roi2, Roi3}. Here, we prove Theorem \ref{thm:main} using combinatorial methods by showing 
that the sum on the right hand side of Equation \eqref{eq:main} takes the same values as $\chi_\lambda$ in Equation \eqref{eq:chi}. To do this, 
we use a relationship between $\lambda$-unimodal permutations and primitive words developed in \cite{archer}.

In Section \ref{sec:background}, we introduce necessary background information and previously known results. In Section \ref{sec:necklaces}, we introduce a correspondence between necklaces and permutations adapted from \cite{archer}. In Section \ref{sec:pf}, we present the proof of the main theorem. 

In Section \ref{sec:cons}, we will see that Theorem \ref{thm:main} implies interesting results about the distribution 
of the descent set on $\C_n$, the set of cyclic permutations. For example, the descent sets of elements of $\C_n$ are equi-distributed with the 
descent sets of the standard Young tableaux that form a basis to the representation described above. 
Additionally, the number of permutations of $\S_{n-1}$ with descent set $D$ is equal to the number of permutations 
of $\C_{n}$ whose descent set is either $D$ or $D \cup [n-1]$. Different proofs of these consequences can 
alternatively be found in \cite{note} and \cite{sergi}, respectively.

Finally, there is an interesting special case of Theorem \ref{thm:main} when $\lambda = (n)$. This gives the result
$$ \sum_{\substack{\pi \in \C_n \\ \text{unimodal}}} (-1)^{\des(\pi)} = \mu(n),$$ of which the author knows of no other proof.

\section{Background}\label{sec:background}

\subsection{Definitions and notation}
Let $\S_n$ denote the set of permutations on $[n] = \{1,2,\ldots, n\}$. We write permutations in their one-line notation as 
$\pi = \pi_1\pi_2\cdots \pi_n$. A \emph{cyclic} permutation is a permutation $\pi \in \S_n$ which can be written in its cycle notation as a single $n$-cycle. For example, the permutation $\pi= 36578124$ is cyclic since it can be written in cycle notation as $\pi = (13584726)$. We denote the set of cyclic permutations of size $n$ by $\C_n$. 

  We say that a word $x_1 x_2 \cdots x_n$ is \emph{unimodal} if  there is some 
$m$, with $1\leq m \leq n$, for which 
\[
	x_1<x_2 < \cdots< x_m>x_{m+1}> \cdots >x_n.
\]
That is, the word is increasing, then decreasing. For example, the word $367841$ is unimodal. 

A \emph{composition of $n$} is an ordered list of non-negative integers $\lambda = (\lambda_1, \lambda_2, \ldots,\lambda_k)$ so that 
$\lambda_1+ \lambda_2+\cdots + \lambda_k = n$. For the remainder of the paper, we will assume that $\lambda$ 
is a composition of $n$ of length $k$. We say a permutation $\pi \in \S_n$ is 
\emph{$\lambda$-unimodal}\footnote{In \cite{char}, these permutations are called \emph{$\mu$-unimodal permutations}. 
In this paper, we use $\lambda$ for the composition and reserve $\mu$ for the M\"{o}bius function. Additionally, in \cite{char}, these permutations are defined slightly differently, where unimodal is taken to mean decreasing, then increasing and thus the permutations in this paper are exactly the complements of the permutations in \cite{char} (See, for example, Prop 2.7 of \cite{archer}). } if when one 
breaks $\pi$ into contiguous segments of lengths $\lambda_i$, each segment is unimodal. That is, defining the partial 
sums of $\lambda$ by $s_i(\lambda) = \lambda_1 + \lambda_2 + \cdots + \lambda_i$,  the segment 
$\pi_{s_{i-1}(\lambda) + 1} \ldots \pi_{s_{i}(\lambda)}$ of $\pi$ is unimodal for all $1\leq i \leq k$.  See Figure \ref{Fig:1} for an example of a $\lambda$-unimodal permutation.

\begin{figure}
\centering
\begin{tikzpicture}[scale = .85]
\draw [fill = white] (0,0) rectangle (9,9);
\draw[help lines, line width = 1.5pt, color = black] (0,0) grid (9,9);
\draw [line width = 2pt, color = blue!50!black] (0,0) -- (1,1);
\draw [line width = 2pt, color = blue!50!black] (0,1) -- (1,0);
\draw [line width = 2pt, color = blue!50!black] (1,3) -- (2,4);
\draw [line width = 2pt, color = blue!50!black] (1,4) -- (2,3);
\draw [line width = 2pt, color = blue!50!black] (2,8) -- (3,9);
\draw [line width = 2pt, color = blue!50!black] (2,9) -- (3,8);
\draw [line width = 2pt, color = blue!50!black] (3,6) -- (4,7);
\draw [line width = 2pt, color = blue!50!black] (3,7) -- (4,6);
\draw [line width = 2pt, color = blue!50!black] (4,4) -- (5,5);
\draw [line width = 2pt, color = blue!50!black] (4,5) -- (5,4);
\draw [line width = 2pt, color = blue!50!black] (5,2) -- (6,3);
\draw [line width = 2pt, color = blue!50!black] (5,3) -- (6,2);
\draw [line width = 2pt, color = blue!50!black] (6,1) -- (7,2);
\draw [line width = 2pt, color = blue!50!black] (6,2) -- (7,1);
\draw [line width = 2pt, color = blue!50!black] (7,7) -- (8,8);
\draw [line width = 2pt, color = blue!50!black] (7,8) -- (8,7);
\draw [line width = 2pt, color = blue!50!black] (8,5) -- (9,6);
\draw [line width = 2pt, color = blue!50!black] (8,6) -- (9,5);
\end{tikzpicture}
\caption{The 
permutation $149753286$ is $(6,3)$-unimodal since $149753$ and $286$ are unimodal segments of lengths 
6 and 3, respectively. The choice of $\lambda=(6,3)$ is not unique here. For example, this permutation is also $(7,2)$-unimodal and $(6, 2, 1)$-unimodal.}
\label{Fig:1}
\end{figure}
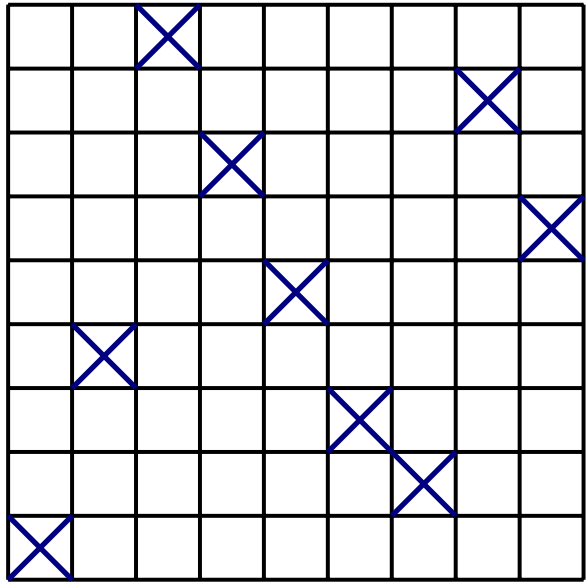
 
A \emph{word} of length $n$ on $m$ letters is a string $s = s_1s_2\ldots s_n$ where $s_i \in \{0,1,\ldots, m-1\}$. 
A \emph{necklace} of length $n$ on $m$ letters is an equivalence class of words $[s]$ so that 
$t = t_1t_2\ldots t_n \sim s = s_1s_2\ldots s_n$ if and only if $t_1t_2\ldots t_n = s_is_{i+1} \ldots s_ns_1\ldots s_{i-1}$ 
for some $1\leq i \leq n$, that is, $t$ is some cyclic rotation of $s$. For example, $1101 \sim 1011\sim 0111\sim 1110$. 
Denote by $W_m(n)$ the set of words of length $n$ on $m$ letters and by $N_m(n)$ the set of necklaces of length 
$n$ on $m$ letters.
 
We call a word $s$ (or a necklace $[s]$) \emph{primitive} if there is no strictly smaller word $q$ so that $s = q^r$ 
for some $r>1$, where $q^r$ denotes the concatenation of $q$ with itself $r$ times. We denote the number of 
primitive necklaces of length $n$ on $m$ letters by $L_m(n)$. We let $a_t(s) = |\{j \in [n]  : s_j = t\}|$, that is the 
number of copies of $t$ in word $s$ and we let $o(s) = \sum_{\text{odd } t} a_t(s)$, that is the number of odd 
letters in word $s$. We denote by $L(a_0, a_1, \ldots, a_{m-1})$ the size of the set of primitive necklaces $[s]$ so 
that $a_t(s) = a_t$. The enumeration of this set where $\sum_i a_i = n$ is well-known. 

\begin{lemma} 
	For $n\geq 1$, suppose that $a_0+ a_1+ \cdots a_{m-1} = n$ and that $a_i\geq 0$ for all $1\leq i \leq m$. Then,
	\begin{equation}
	\label{eq:L}
		L(a_0, a_1\ldots, a_{m-1}) = 
		\frac{1}{n}\sum_{\ell \mid\gcd(a_0, \ldots, a_{m-1})}
		\mu(\ell) \frac{(n/\ell)!}{(\frac{a_0}{\ell})!(\frac{a_1}{\ell})! \cdots(\frac{a_{m-1}}{\ell})!}.
	\end{equation}
\end{lemma}
 
The $\lambda$-unimodal permutations are special members of the alternating juxtaposition class, that is, permutations which are composed of $k$ segments that alternate increasing and decreasing. In \cite{archer}, the authors use infinite periodic words (equivalently necklaces) to obtain permutations in a given juxtaposition class via maps called signed shifts. Using these methods as motivation, we will define a set of necklaces $N_\lambda$ and a map $\PPat_\lambda$ to obtain $\lambda$-unimodal permutations from necklaces.

Define the set $N_\lambda \subseteq N_{2k}(n)$ to be the set of necklaces $[s]$ so that 
$a_{2t}(s) + a_{2t+1}(s) = \lambda_{t+1}$ for all $0\leq t \leq k-1$ and so that $[s]$ is either primitive or $s = q^2$ 
for some primitive word $q$ so that $o(q)$ is odd. For example, the word $s = 00121\in N_{(4,1)}$ since 
$a_0(s) + a_1(s) = 2+2 = 4$ and $a_2(s) + a_3(s) = 1+0 = 1$ and $s$ is primitive. For another example, 
$t = 0213302133$ is in $N_{(4,6)}$ since $a_0(t) + a_1(t) = 2+2 = 4$ and $a_2(t) + a_3(t) = 2+4 = 6$ and also 
$t = (02133)^2$ where $o(02133) = 3$ and $02133$ is primitive.

Let $N_\lambda^{(m)}$ be the set of elements $[s] \in N_\lambda$ where $o(s) = m$. Denote by 
$\bigL{\lambda}{m}$ the number of primitive necklaces in $N_\lambda^{(m)}$. We will often use the notation $d = \gcd(\lambda_1, \ldots, \lambda_k)$. If $2\mid d$, let $\lambda/2$ denote the composition of $n/2$ obtained by dividing each part of $\lambda$ by 2.

\begin{lemma} 
\label{lem:N=bigL}
	For $m, n \geq 1$ and $d = \gcd(\lambda_1, \ldots, \lambda_k)$. 
	\[
		\bigl| N_\lambda^{(m)} \bigr| = \begin{cases} \bigL{\lambda}{m} + \bigL{\lambda/2}{m/2} & 
		\text{if } d \text{ is even and } m \equiv 2\bmod 4 \\ \bigL{\lambda}{m} & \text{otherwise.}\end{cases} 
	\]
\end{lemma}
\begin{proof}
	The elements of $N_\lambda^{(m)}$ which are not primitive are exactly those of the form $[s]$ where $s = q^2$ 
	and $q$ is primitive with $o(q) = m/2$ odd. Therefore, we must have that $n$ and $d$ are even and $m \equiv 2 \bmod 4.$ 
	Conversely, if $n$ and $d$ are even and $m \equiv 2 \bmod 4$, given any primitive word $q$ of length $n/2$ with $a_{2t}(q) + a_{2t+1}(q) = \lambda_{t+1}/2$ 
	and $o(q) = m/2$, we will have $[q^2] \in N_\lambda^{(m)}.$
\end{proof}

Notice that, by definition, we can write $\bigL{\lambda}{m}$ in the following way:
\begin{equation}
\label{eq:bigL}
	\bigL{\lambda}{m} =  \sum_{\substack{\sum_t i_t = m \\ 0\le i_t\leq \lambda_t}}
	L(\lambda_1-i_1,i_1, \ldots, \lambda_k-i_k, i_k).
\end{equation}
The next lemma demonstrates a useful symmetry of $\bigL{\lambda}{m}$.
\begin{lemma}
\label{lem:bigL} 
	For all $0\leq m \leq n$ and any composition $\lambda$ of $n$, 
	\[
		\bigL{\lambda}{m} = \bigL{\lambda}{n-m}.
	\]
\end{lemma}
\begin{proof}
	Given a primitive word $s \in N_\lambda$ with $o(s) = m$, we can construct a primitive word $s' \in N_\lambda$ 
	with $o(s) = n-m$ by letting $s'_i = 2t$ if $s_i = 2t+1$ and letting $s'_i = 2t+1$ if $s_i = 2t$ for all $0\leq t \leq k-1$. 
	Doing this switches odd letters with even letters and also ensures 
	$\lambda_{t+1} = a_{2t}(s) + a_{2t+1}(s) = a_{2t}(s') + a_{2t+1}(s')$.
\end{proof}
 
\subsection{The $\PPat_\lambda$ mapping}

In order to prove Theorem \ref{thm:main}, we must define a mapping $\PPat_\lambda$ from $N_\lambda$ to 
$\lambda$-unimodal cyclic permutations. The mapping we describe here is a special case of the mapping 
$\PPat_\sigma$ defined in \cite{archer}.

Define a map $\Sigma: W_{2k}(n) \to W_{2k}(n)$ which takes a word $s_1s_2\ldots s_n$ to the word 
$s_2s_3\ldots s_ns_1$. We will define an ordering on words in $W_{2k}(n)$ denoted by $\prec$. This ordering is similar to $\prec_\sigma$ for $\sigma = (+-)^k$ found in \cite{archer}, where $\prec_\sigma$ is a generalization of the lexicographical ordering which is motivated by dynamics of the signed shifts. 

Suppose that 
$s= s_1s_2\ldots s_n$ and $s' = s'_1s'_2\ldots s'_n$ and that for some $1\leq i \leq n$, we have 
$s_1\ldots s_{i-1} = s'_1\dots s'_{i-1}$ and $s_i \neq s'_i$. Then we say that $s\prec s'$ if either $o(s_1\ldots s_{i-1})$ 
is even and $s_i<s'_i$ or if $o(s_1\ldots s_{i-1})$ is odd and $s_i>s'_i$, where $<$ is the ordering on $\{0, 1, 2, \ldots, 2k-1\}$ inherited from the integers. 

We are now prepared to define a mapping $\PPat_\lambda: N_\lambda \to \C_n$. (We will see in Prop \ref{lem:map PPat} that the image of this map is actually $\C_\lambda$, the set of $\lambda$-unimodal cyclic permutations.) Suppose first that $[s]\in N_\lambda$ is primitive. Choose a representative 
$s\in[s]$. Take $\pi = \pi_1\pi_2\ldots \pi_n\in \S_n$ to be the permutation (which we call the \emph{pattern}) so that 
$\pi$ is in the same relative order as the sequence 
\[
	s, \Sigma(s), \Sigma^2(s), \ldots, \Sigma^{n-1}(s)
\] 
with respect to the ordering $\prec$. Then take $\pihat = (\pi_1\pi_2\ldots \pi_n)$, the cyclic permutation obtained by 
sending $\pi_1$ to $\pi_2$, $\pi_2$ to $\pi_3$, etc. Then we say that $\PPat_\lambda([s]) = \pihat$. Notice that the 
choice of representative $s \in [s]$ does not matter.

If $[s]\in N_\lambda$ is not primitive, then $s$ is of the form $q^2$ for primitive $q$ where $o(q)$ is odd. Defining the 
pattern $\pi$ in this case is a little different since for any representative $s\in[s]$, 
$\Sigma^{i}(s) = \Sigma^{\frac{n}{2}+i}(s)$ and so it is not immediately obvious how to associate a permutation of 
length $n$ to it. Though $\prec$ does not give an ordering on the cyclic shifts of $s$, it does give a partial ordering where the only pairs which cannot be compared are $\Sigma^i(s)$ and $\Sigma^{\frac{n}{2}+i}(s)$ since they are equal for all $0\leq i <\frac{n}{2}$. If we wish to assign a permutation to the cyclic shifts of $s$ which reflect the ordering, then for any $i$, we should have that either $\pi_i =\pi_{\frac{n}{2} +i} + 1$ or $\pi_i =\pi_{\frac{n}{2} +i} - 1$. We will decide to take $\pi_1<\pi_{\frac{n}{2} +1}$. This first assignment is made somewhat arbitrarily, but regardless will return the same $\pihat$ for any representative $s\in [s]$. By making this assignment, we extend our partial order determined by $\prec$ to a linear order on the cyclic shifts of $s$ and we are thus forced to have $\pi_i<\pi_{\frac{n}{2} +i}$ whenever $o(s_1\ldots s_{i-1})$ is even and  $\pi_i>\pi_{\frac{n}{2} +i}$ whenever $o(s_1\ldots s_{i-1})$ is odd.


Let us see an example. Suppose $\lambda = (3,6)$ and $s = 321132202$ is our choice of representative. Certainly, $\pi_8 =1$ since $\Sigma^7(s) = 023211322$ must be less than any other cyclic shift of $s$ since it is the only one that starts with 0 (the smallest number in our set $\{0,1,2,3\}$). Additionally, $\pi_3$ and $\pi_4$ must be 2 and 3 in some order since $\Sigma^2(s)$ and $\Sigma^3(s)$ both start with 1. To determine which one is smaller, we compare them using the ordering $\prec$. $\Sigma^2(s) = 113220232$ and $\Sigma^3(s) = 132202321$. Comparing the two, we see that the first place they disagree is the second position and $o(1)$ is odd. By comparing the values at the first place they disagree, we determine that $\pi_3>\pi_4$. By repeating this process, we obtain $\pi = 953286417$ and $\pihat = (953286417) = 782134965$. Notice that $\pihat$ is $(3,6)$-unimodal since $782$ and $134965$ are 
both unimodal segments. If we had chosen a different representative, for example, $s = 322023211$, then we would have obtained $\pi = 864179532$ and $\pihat = (864179532) = 782134965$. Indeed, the choice of representative did not matter.

For an example when $[s]$ is not primitive, consider $\lambda = (4,4)$ and $s = 02210221$. Notice that 
$s = (0221)^2$ where $0221$ is primitive and $o(0221) = 1$ is odd. Therefore $[s] \in N_\lambda$. In this case, we notice that $s = \Sigma^4(s)$ is the smallest with respect to $\prec$ and so $\pi_1$ and $\pi_5$ are 1 and 2 in some order. We arbitrarily assign $\pi_1 = 1$ and $\pi_5 = 2$ (this is the only time we make a ``choice''). $\Sigma^{3}(s)=\Sigma^7(s)$ are the second smallest with respect to $\prec$ and thus $\pi_4$ and $\pi_8$ will be 3 and 4 in some order. Since we chose $\pi_1<\pi_5$, we are forced to have $\pi_4<\pi_8$. We continue until we obtain 
$\pi = 17532864$ and $\pihat = (17532864) = 78213456$. Notice that $\pihat$ is $(4,4)$-unimodal since $7821$ 
and $3456$ are both unimodal segments.

\begin{remark}
\label{remark: prec}
	Notice that the way $\prec$ is defined, we must have that if $s\prec s'$ and $s_1 = s'_1$, then 
	$\Sigma(s)\prec\Sigma(s')$ if $s_1$ is even and $\Sigma(s)\succ\Sigma(s')$ if $s_1$ is odd.
\end{remark}

\subsection{The character $\chi$} \label{sec:chi}

Consider the one-dimensional representation $\psi$ on $H = \langle (123\cdots n) \rangle$ obtained by letting 
$\psi(g^i) = \zeta_n^i$ where $g$ is a generator of $H$ and $\zeta_n$ is a primitive $n$th root of unity. We also use $\psi$ 
to denote the character for this representation since it coincides with the representation itself. Denote by $\rho$ the 
representation on $\S_n$ induced from the representation on $H\leq \S_n$ described above. In the next proposition, 
we prove that $\chi$, the character of $\rho$, satisfies Equation \eqref{eq:chi} on each conjugacy class of 
cycle type $\lambda$ of $\S_n$. Recall that $\chi_\lambda$ denotes the value of $\chi$ on conjugacy class $\lambda.$

The following proposition is well known (see, for example, \cite{gill}). 

\begin{proposition}
\label{prop:ind}
	For any conjugacy class 
	$\lambda$, 
	\[
		\chi_\lambda = \begin{cases} (k-1)! d^{k-1} \mu(d) & \text{if } \lambda = (d^k) \\ 0 & \text{otherwise.} \end{cases}
	\]
\end{proposition}
\begin{proof}
	The induced character on $\S_n$ is defined by 
	\[
		\chi(\tau) = \frac{1}{|H|} \sum_{\sigma \in \S_n} \psi(\sigma^{-1}\tau\sigma)
	\]
	where we define $\psi(\sigma^{-1}\tau\sigma) = 0$ when $\sigma^{-1}\tau\sigma \notin H$. 
	
	Since $\sigma^{-1}\tau\sigma$ is in the same conjugacy class as $\tau$, it will have the same cycle type. Notice 
	that $H$ only contains permutations with cycle type $(d^k)$ for some $k$. Therefore, if $\tau$ has cycle type 
	which is not of the form $(d^k)$, then every $\psi(\sigma^{-1}\tau\sigma) $ contributes 0 and so $\chi(\tau) = 0$. 
	Thus, we see that $\chi_\lambda = 0$ whenever $\lambda \neq (d^k)$ for some $k$. 
	
	It is well known that 
	\[
		\sum_{\gcd(i,n) = 1} \zeta_n^i = \mu(n).
	\]
	Suppose $\tau$ has cycle type $\lambda = (d^k)$. Then $\psi(\sigma^{-1}\tau\sigma)$ will always be either $0$ or 
	a primitive $d^{\text{th}}$ root of unity since permutations in $H$ with cycle type $d^k$ are the generators of the 
	cyclic subgroup of $H$ of size $d$. Since $|H|= n = dk$ and $\sum_{\gcd(d,i) = 1} \zeta_d^i = \mu(d)$, it is enough 
	to show that $|\{ \sigma \in \S_n : \sigma^{-1}\tau\sigma = \tau'\}| = k! d^k$ for each $\tau' \in H$ with cycle type 
	$(d^k)$. 
	
	Suppose we have two elements $g$ and $h$ so that $g^{-1} \tau g = \tau'$ and $h^{-1} \tau h = \tau'$. Then we 
	must have that $g^{-1} \tau g  = h^{-1} \tau h$, which means that $gh^{-1}\tau hg^{-1} = \tau$ and so $hg^{-1}$ is 
	in the centralizer of $\tau$, $C(\tau)$. Therefore, if $g$ is such that $g^{-1} \tau g = \tau'$, then every possible $h$ 
	such that $h^{-1} \tau h = \tau'$ must be of the form $ag$ for some $a$ in $C(\tau)$. Therefore 
	$|\{ \sigma \in \S_n : \sigma^{-1}\tau\sigma = \tau'\}| = |C(\tau)|$. By the orbit-stabilizer theorem, this is equal to 
	$|\S_n|/|\{\pi \in \S_n : \pi \text{ has cycle type } (d^k)\}$. Using the well-known formula for the size of conjugacy 
	classes in $\S_n$, we find that indeed $|C(\tau)| = k! d^k$. 
\end{proof}

\subsection{Counting Lemmas}

Here we include a few combinatorial lemmas we will need in the proof of Theorem \ref{thm:main} in Section \ref{sec:pf}.

\begin{lemma}
\label{lem:bincoef}
	Suppose $p$ is an $r$-degree polynomial. Then when $0\leq r<n$, 
	\begin{equation}
	\label{eq:p_r}
		\sum_{i = 0}^{n} (-1)^{i} p(i) {{n}\choose{i}}  = 0.
	\end{equation}
\end{lemma}
\begin{proof}
	Given the Binomial Theorem, 
	\[
		(x+1)^n = \sum_{i = 0}^n\binomial{n}{i} x^i,
	\]
	if $r<n$, then we can take $r$ derivatives of both sides to obtain: 
	\[
		\frac{n!}{(n-r)!}(x+1)^{n-r} = \sum_{i = r}^n\binomial{n}{i} x^{i-r} p_r(i), 
	\]
	where $p_r(i) = i(i-1)\cdots (i-r+1)$. Plugging in $x = -1$, we obtain Equation \eqref{eq:p_r} for each $p_r$. 
	(Since $p_r(i) = 0$ for all $i<r$, we can start indexing at $i = 0$.) Any polynomial can be written as a linear
	combination of these polynomials and so Equation \eqref{eq:p_r} must hold for all polynomials of degree 
	at most $n-1$.
\end{proof}

\begin{lemma}
\label{lem:incex} 
	For any $d,k\geq 1$, 
	\[
		\sum_{i = 1}^{k} (-1)^{i+k} {{di}\choose{k}} {{k}\choose{i}} = d^k.
	\]
\end{lemma}
\begin{proof}
	We claim that both sides of this formula count the number of ways to pick a single element from each of $k$ 
	different boxes containing $d$ objects each. The right hand side clearly counts the ways to do this. On the left 
	hand side, we choose $i$ boxes to consider, in ${{k}\choose{i}}$ ways, then choose $k$ elements from this collection 
	of boxes (possibly pulling many from the same box) in ${{di}\choose{k}}$ ways. Using inclusion exclusion, we find 
	the number of ways to choose $k$ objects from $k$ different boxes.
\end{proof}

\begin{lemma}
\label{lem:long} 
	Suppose $\gamma_1 + \gamma_2 + \cdots + \gamma_k = r$. Then, 
	\[
		\sum_{\substack{0\le a_t\leq \gamma_t  \\ a_1+\cdots +a_k = i}}
		\frac{(r)!}{ (\gamma_1-a_1)!(a_1)!\cdots (\gamma_k-a_k)!(a_k)!} = \frac{(r)!}{\gamma_1!\cdots \gamma_k!} 
		{{r}\choose{i}}.
	\]
\end{lemma}
\begin{proof}
	Both sides count the number of ways to separate $r$ objects into $k$ different sets of size $\gamma_t$ for 
	$1\leq t \leq k$ and then color $i$ of the objects. 
\end{proof}

\section{Necklaces and cyclic permutations}
\label{sec:necklaces}

The following theorem is a special case of Theorem 2.1 in \cite{archer}. For that reason, we provide a short proof without 
all of the details. A complete proof of a more general statement can be found in \cite{archer}.

\begin{lemma}
\label{lem:map PPat}
	For any $[s] \in N_\lambda$, we have $\PPat_\lambda([s]) \in \mathcal{C}(\lambda)$. Additionally, the map 
	$\PPat_\lambda: N_\lambda \to \mathcal{C}(\lambda)$ is surjective.
\end{lemma}
\begin{proof}
	Suppose $[s] \in N_\lambda$, $\pi$ is the pattern of $s'\in [s]$, and $\pihat = \PPat_\lambda([s])$. For all 
	$0\leq t \leq 2k$, define $e_t = |\{j \in [n] : s'_j <t \}|$. For any $0\leq t <2k$, suppose
	$e_{t} < \pi_i<\pi_j\leq e_{t+1}$. Since $\pi_i<\pi_j$, it must be true that 
	$\Sigma^{i-1}(s')\prec\Sigma^{j-1}(s')$. Additionally, since $e_{t} < \pi_i,\pi_j\leq e_{t+1}$, we must have 
	$s'_i = s'_j = t$. By Remark \ref{remark: prec}, it follows that $\Sigma^i(s')\prec\Sigma^j(s')$ (and thus $\pi_{i+1}<\pi_{j+1}$) 
	if $t$ is even and $\Sigma^i(s')\succ \Sigma^j(s')$ (and thus $\pi_{i+1}>\pi_{j+1}$) if $t$ is odd. 
	
	From this, it will follow that the segment $\pihat_{e_t+1}\ldots \pihat_{e_{t+1}}$ is increasing if $t$ is even 
	and decreasing if $t$ is odd. For if $e_{t} < a<b\leq e_{t+1}$, we only need to show that $\pihat_a<\pihat_b$ if $t$ 
	is even and $\pihat_a>\pihat_b$ if $t$ is odd. First, notice that $\pihat_{\pi_i} = \pi_{i+1}$. Therefore, take $i$ and 
	$j$ so that $\pi_i = a$ and $\pi_j= b$. Then $e_{t} < \pi_i<\pi_j\leq e_{t+1}$, and thus 
	$\pihat_a = \pi_{i+1}<\pi_{j+1} = \pihat_b$ if $t$ is even and $\pihat_a = \pi_{i+1}>\pi_{j+1}= \pihat_b$ if $t$ is odd.

	Finally, since $e_{t} = \sum_{r<t} a_r(s)$, it follows that $e_{2t+2}-e_{2t} = a_{2t}(s) + a_{2t+1}(s) = \lambda_{t+1}$. 
	Therefore the segment $\pihat_{e_{2t}+1} \ldots\pihat_{e_{2t+1}} \pihat_{e_{2t+1}+1}\ldots \pihat_{e_{2t+2}}$ has 
	length $\lambda_{t+1}$ and is unimodal for all $0\leq t<k$. Therefore $\pihat \in \C_\lambda$. 
	
	To see that the map is surjective let $\pihat \in \C_\lambda$ be arbitrary and let $\pi =\pi_1\ldots \pi_n$ be such that 
	$\pihat = (\pi_1\pi_2\ldots \pi_n)$. Since $\pihat \in \C_\lambda$, there is some sequence 
	$0= e_0\leq e_1\leq \cdots \leq e_{2k} = n$ so that (1) the segment $\pihat_{e_t+1}\ldots \pihat_{e_{t+1}}$ is 
	increasing if $t$ is even and decreasing if $t$ is odd, and (2) $e_{2t+2}-e_{2t} =  \lambda_{t+1}$ for all 
	$0\leq t \leq k-1$. The word $s = s_1s_2\ldots s_n$ one obtains by setting $s_i = t$ if $e_t<\pi_i\leq e_{t+1}$ is a 
	necklace in $N_\lambda$ so that $\PPat_\lambda([s]) = \pihat$. Therefore, the map is surjective.
\end{proof}

\begin{remark}
	It is a nontrivial fact that any choice of $0= e_0\leq e_1\leq \cdots \leq e_{2k} = n$ will result in a primitive 
	word (or a 2-periodic word) and thus lies in $N_\lambda$. The proof of this remark is omitted from the proof above, but follows from a case-by-case analysis.
\end{remark}

The next theorem describes the relationship between the number of odd letters in an element of $N_\lambda$ and the 
number of descents of its image under $\PPat_\lambda$. This will prove useful when we rewrite the sum in Theorem 
\ref{thm:main} as a sum over necklaces.

Recall that $N_\lambda^{(m)}$ denotes the set of elements $[s] \in N_\lambda$ where $o(s) = m$. Let 
$\PPat_\lambda^{(m)}$ be the map $\PPat_\lambda$ restricted to the set $N_\lambda^{(m)}$. Also, let $\C_\lambda(m)$ be 
the set of $\lambda$-unimodal cycles $\tau$ with $|\Des(\tau)\setminus S(\lambda)| = m$, where $S(\lambda)$ is the 
set of partial sums of $\lambda$, and let $\al{\lambda}{m} = |\C_\lambda(m)|$. 

\begin{lemma}
	The map 
	\[
		\PPat_\lambda^{(m)}: N_\lambda^{(m)} \to \bigcup_{j = 0}^{\min\{k,m\}}\C_\lambda(m-j)
	\]
	is surjective. Moreover, for each $\tau \in \C_\lambda(m-j)$, the size of the preimage 
	$\big(\PPat_\lambda^{(m)}\big)^{-1}(\tau)$ is $\binomial{\min\{k,m\}}{j}$. 
\end{lemma}
\begin{proof}
	Let $\tau \in C_\lambda(m-j)$. For some $\pi = \pi_1\pi_2\ldots \pi_n$, we have $\tau = \pihat$. By Lemma \ref{lem:map PPat},   
	there is a word $s = s_1s_2\ldots s_n \in N_\lambda$ so that $\PPat([s]) = \pihat$. 
	We obtain this word by finding a sequence $0= e_0\leq e_1\leq \cdots \leq e_{2k} = n$ so that the segment 
	$\pihat_{e_t+1}\ldots \pihat_{e_{t+1}}$ is increasing if $t$ is even and decreasing if $t$ is odd. However, this 
	sequence is not unique. Certainly, we must have that  $e_{2t+2}-e_{2t}= \lambda_{t+1}$ for all $0\leq t \leq k-1$ 
	and so $e_{2t}$ is fixed for every $ 0 \leq t \leq k$. However, for $e_{2t+1}$, we have exactly two choices for all 
	$0\leq t \leq k-1$. This is because the unimodal segment of length $\lambda_t \neq 0$ can be broken up in 
	exactly two ways, where the peak could be included in either the increasing segment or the decreasing segment. 
	
	Suppose $\tau = \pihat$ has exactly $d$ descents, not including those from $S(\lambda)$. Then we must have 
	at least $d$ odd letters. For example, 
	in a unimodal permutation $24587631$, there are 4 descents and the length of the decreasing segment is either 4 or 5. 
	Choose $j$ of the $k$ peaks to belong to decreasing part of the unimodal segments. This will add exactly 
	$j$ odd letters to the minimum number, resulting in $m=d+j$ odd letters in a given representative $s\in[s]$. 
	Clearly, there are $\binomial{\min\{k,m\}}{j}$ ways to choose $j$ peaks.
\end{proof}

\begin{corollary}
\label{cor:Nlambda}
	Suppose $\lambda$ is a composition of $n$ with $k$ parts and $0\leq m\leq n$. Then,
	\[
		\bigl| N_\lambda^{(m)} \bigr| = \sum_{j = 0}^k \binomial{k}{j} \al{\lambda}{m-j}.
	\]
\end{corollary}

Using the above relationship, we can use generating functions to find an equation for $\al{\lambda}{m}$ in terms of 
$\bigl| N_\lambda^{(m)} \bigr|$. 

\begin{lemma} 
\label{lem:2}
	For a composition $\lambda$ of $n$ with $k$ parts and for $0\leq m\leq n$, 
	\begin{equation}
	\label{eq:odd al} 
		\al{\lambda}{m} =  \sum_{j=0}^{m} (-1)^{m-j} {{m-j+k-1}\choose{k-1}} \bigl| N_\lambda^{(j)} \bigr|.
	\end{equation}
\end{lemma}
\begin{proof}
	We have a formula for $\bigl| N_\lambda^{(m)} \bigr|$ in terms of $\al{\lambda}{m}$ by Corollary~\ref{cor:Nlambda}. Since 
	${{k}\choose{j}} =0$ when $j>k$, we can rewrite it as 
	\[
		\bigl| N_\lambda^{(m)} \bigr| = \sum_{j = 0}^m {{k}\choose{j}} \al{\lambda}{m-j}.
	\]
	We can write this relationship in terms of the generating functions for $\bigl| N_\lambda^{(m)} \bigr|$ and $\al{\lambda}{m}$.
	\[	
		\sum_{m\geq 0} \bigl| N_\lambda^{(m)} \bigr| x^m= \Bigg(\sum_{j=0}^k {{k}\choose{j}}x^j\Bigg) 
		\Bigg(\sum_{m\geq 0} \al{\lambda}{m}x^m\Bigg) =(1+x)^k \Bigg(\sum_{m\geq 0} \al{\lambda}{m}x^m\Bigg).
	\]
	By multiplying both sides of this equation by $(1+x)^{-k}$, it follows that
	\[
		\sum_{m\geq 0} \al{\lambda}{m} x^m= \Bigg(\sum_{j\geq 0} (-1)^j {{j+k-1}\choose{k-1}} x^j\Bigg) 
		\Bigg(\sum_{m\geq 0}  \bigl| N_\lambda^{(m)} \bigr| x^m\Bigg).
	\] 
	Equation~\eqref{eq:odd al} follows. 
\end{proof}

\section{Proof of the main result}\label{sec:pf}

Notice that using Equation \eqref{eq:chi} and the definition of $\al{\lambda}{m}$, we can rewrite the statement of 
Theorem \ref{thm:main} in the following way:
\begin{equation}
\label{eq:re}
	\sum_{m= 0}^{n-k} (-1)^m \al{\lambda}{m} =  \begin{cases}  (k-1)! d^{k-1} \mu(d)& \text{if }\lambda = (d^k), \\ 
	0 & \text{otherwise.} \end{cases}
\end{equation}
Denoting $d = \gcd(\lambda_1,\lambda_2,\cdots, \lambda_k)$, we will prove two cases in the next two theorems, when $d$ is odd 
and when $d$ is even. Combining the following Theorems \ref{thm:dodd} and \ref{thm:deven} will give us a proof of Equation 
\eqref{eq:re} and thus a proof of Theorem \ref{thm:main}.
\begin{theorem}
\label{thm:dodd}
	If $d = \gcd(\lambda_1,\cdots,\lambda_k)$ is odd, then
	\begin{equation}
	\label{eq:dodd} 
		\sum_{m = 0}^{n-k} (-1)^m \al{\lambda}{m}  = \begin{cases}  (k-1)! d^{k-1} \mu(d)& \text{if }\lambda = (d^k), \\ 
		0 & \text{otherwise.} \end{cases}
	\end{equation}
\end{theorem}
\begin{proof}
	Using Lemma \ref{lem:2}, we can expand the left hand side of Equation~\eqref{eq:dodd} by plugging in the right 
	hand side of Equation~\eqref{eq:odd al} for $\al{\lambda}{m}$. By Lemma \ref{lem:N=bigL}, we know that when 
	$d$ is odd, $\bigl| N_\lambda^{(j)} \bigr| = \bigL{\lambda}{j}$. We switch the order of summation which allows us 
	to simplify the equation to a single sum. The binomial identity used here to simplify the equation is easily checked. 
	\[
		\sum_{j = 0}^{n-k} \sum_{m=j}^{n-k} (-1)^{j} {{m-j+k-1}\choose{k-1}} \bigL{\lambda}{j} = 
		\sum_{j = 0}^{n-k} (-1)^{j} {{n-j}\choose{k}} \bigL{\lambda}{j}.
	\]
	Notice that since Lemma \ref{lem:bigL} implies $\bigL{\lambda}{n-j} = \bigL{\lambda}{j}$, then we can perform a 
	change of variables by replacing $j$ with $n-j$ to obtain the following formula. We then expand this formula using 
	Equations \eqref{eq:L} and \eqref{eq:bigL}.
	 \begin{align*}
	 	\sum_{j = k}^{n} (-1)^{n-j} &{{j}\choose{k}} \bigL{\lambda}{j} \\
			&= \sum_{j = 1}^{n} (-1)^{n-j} {{j}\choose{k}} \sum_{\substack{0\le i_t\leq \lambda_t \\ i_1+\cdots +i_k = j}}
				\frac{1}{n}\sum_{\ell \mid \gcd(d,i_1, \cdots, i_k)} \mu(\ell) \frac{(n/\ell)!}{(\frac{\lambda_1-i_1}{\ell})! 
				(\frac{i_1}{\ell})!\cdots (\frac{\lambda_k-i_k}{\ell})!(\frac{i_k}{\ell})!}.
	\end{align*}
	Notice that for convenience, we write the sum on the right starting at 1. This does not change the formula 
	since $\binomial{j}{k} = 0$ whenever $j<k$. 
	  
	Since $\ell \mid\gcd(d,i_1,i_2, \cdots, i_k)$, certainly $\ell\mid d$ and thus $\ell$ must always be odd. Also, for 
	any fixed $\ell$ and for any choice of $i_1, \ldots, i_k$, we have that $\ell\mid i_t$ for $1\leq t \leq k$ and thus 
	$\ell \mid j$. We can rewrite the above formula, now moving the sum over $\ell\mid d$ to the front. We make the 
	following substitutions for the indices in our equation, letting $i$, $r$, $a_t$, and $\gamma_t$ be such that 
	$j = \ell i$, $n = \ell r$, $i_t = \ell a_t$, $\lambda_t = \ell\gamma_t$ for all $1\leq t\leq k$. After these substitutions, 
	our resulting formula is the following:
	\[
		\frac{1}{n} \sum_{\ell\mid d}  \mu(\ell)  \sum_{i = 1}^{r} (-1)^{\ell r-\ell i} {{\ell i}\choose{k}} 
		\sum_{\substack{0\le a_t\leq \gamma_t \\ a_1+\cdots + a_k = i}}
		\frac{(r)!}{(\gamma_1-a_1)!(a_1)!\cdots (\gamma_\ell-a_k)!(a_k)!}.
	\]
	We simplify this formula by noticing first that the right-most sum is equal to 
	$\frac{(r)!}{\gamma_1!\cdots \gamma_k!} {{r}\choose{i}}$ by Lemma \ref{lem:long}. Also, since $\ell$ is odd, we 
	have $(-1)^{\ell r - \ell i} = (-1)^{i+r}$ We therefore obtain:
	\[
		\frac{1}{n} \sum_{\ell\mid d}  \mu(\ell) \frac{(r)!}{\gamma_1!\cdots \gamma_k!} \sum_{i = 1}^{r} (-1)^{i+r} 
		{{\ell i}\choose{k}} {{r}\choose{i}}.
	\]
	Since ${{\ell i}\choose{k}}$ is a degree $k$ polynomial in $i$, the right-most sum is zero when $r > k$ by 
	Lemma \ref{lem:bincoef}. When $\frac{n}{\ell} = r \leq k$, we have $n \leq \ell k$, but $\ell \mid d$ and in turn 
	$d \mid \lambda_t$ for all $1\leq t \leq k$ where $\lambda_1+\cdots +\lambda_k = n$. Therefore, we must have 
	that $\ell = d = \lambda_t$ for all $1\leq t \leq k$. It follows that if $\lambda \neq (d^k)$, then 
	$\sum_{m=0}^{n-k} (-1)^m \al{\lambda}{m} = 0$.  
	
	If $\lambda = (d^k)$, then by the above argument, we must have that $r = k$ and $\ell = d$. 
	Substituting these values into the formula gives us:
	\[
		\mu(d)  \frac{(k)!}{k d} \sum_{i = 1}^{k} (-1)^{i+k} {{di}\choose{k}} {{k}\choose{i}}.
	\]
	By Lemma \ref{lem:incex}, we know that $\sum_{i = 1}^{k} (-1)^{i+k} {{di}\choose{k}} {{k}\choose{i}} = d^k.$ 
	Equation \eqref{eq:dodd} follows.
\end{proof}

\begin{theorem}
\label{thm:deven}
	If $d = \gcd(\lambda_1,\cdots, \lambda_k)$ is even, then
	\begin{equation}
	\label{eq:deven}
		\sum_{m = 0}^{n-k} (-1)^m \al{\lambda}{m} = \begin{cases} (k-1)! \cdot d^{k-1} \mu(d) & \text{if } \lambda = (d^k), \\ 
		0 & \text{otherwise.} \end{cases}
	\end{equation}
\end{theorem}
\begin{proof}
	Let $\lambda' = \lambda/2$, the composition of $n/2$ of length $k$ with $\lambda'_i = \lambda_i/2$. We use 
	Lemmas \ref{lem:N=bigL} and \ref{lem:2} to expand Equation \eqref{eq:deven}.
	\begin{equation}
	\label{eq}
		\sum_{m = 0}^{n-k}  \sum_{j=0}^{m} (-1)^{j} {{m-j+k-1}\choose{m-j}} \bigL{\lambda}{j} + 
		\sum_{m = 0}^{n-k} \sum_{\substack{0\le i \leq m \\ i \equiv 2 \bmod 4}} {{m-i+k-1}\choose{m-i}} 
		\bigL{\lambda'}{\frac{i}{2}}.
	\end{equation}
	Notice that the first sum in Equation \eqref{eq} looks similar to the formula in the proof of 
	Theorem \ref{thm:dodd}. Through the same steps, we find that it can be 
	written as:
	\[
		\frac{1}{n} \sum_{\ell\mid d}  \mu(\ell) \frac{(r)!}{\gamma_1!\cdots \gamma_k!} 
		\sum_{i = 1}^{r} (-1)^{\ell i} {{\ell i}\choose{k}} {{r}\choose{i}}.
	\]
	From this point, the proof is different from the proof of Theorem \ref{thm:dodd}, because $\ell$ can be even. 
	We split the sum into the two cases when $\ell$ is either even or odd.
	\[
		\frac{1}{n} \sum_{\substack{\ell\mid d\\ \ell \text{ even}}}  \mu(\ell) \frac{(r)!}{\gamma_1!\cdots \gamma_k!} 
		\sum_{i = 1}^{r} {{\ell i}\choose{k}} {{r}\choose{i}}  + \frac{1}{n} 
		\sum_{\substack{\ell\mid d\\ \ell \text{ odd}}} \mu(\ell) \frac{(r)!}{\gamma_1!\cdots \gamma_k!} 
		\sum_{i = 1}^{r} (-1)^{i} {{\ell i}\choose{k}} {{r}\choose{i}}. 
	\]
	As before, by Lemma \ref{lem:bincoef} the second sum is zero except possibly when $r = k$ and $\ell=d$. 
	However $d$ is even and $\ell$ is odd, so this can never be the case. Therefore, the second sum must always 
	be zero. For reasons we'll see later, let $n' = n/2$, $d' = d/2$ and $\ell' = \ell/2$. We only need to consider $\ell'$ odd, 
	since if it were even, $\mu(\ell) = 0$. If $\ell'$ odd, then we have that $\mu(\ell) = \mu(2)\mu(\ell') = -\mu(\ell')$. 
	Therefore, the first sum in Equation \eqref{eq} can now be written as:
	\begin{equation}
	\label{eq:lhs}
		-\frac{1}{n} \sum_{\substack{\ell'\mid d'\\ \ell' \text{ odd}}}  \mu(\ell') \frac{r!}{\gamma_1!\cdots \gamma_k!} 
		\sum_{i = 1}^{r} {{2\ell'i}\choose{k}} {{r}\choose{i}}. 
	\end{equation}
	Now, consider the second sum in equation \eqref{eq}. 
	As usual, we change the order of summation to simplify the equation to one summation. For simplicity, we change 
	our variable from $j$ to $(n-i)/2$. Let $n' = n/2$. Using Lemma \ref{lem:bigL}, which says that 
	$\bigL{\lambda'}{j} = \bigL{\lambda'}{n'-j}$, we obtain the following formula:
	\[
		\sum_{\substack{0\le j \leq n' \\ n'-j \text{ odd}}} {{2j}\choose{k}} \bigL{\lambda'}{j}.
	\]
	There are two cases: when $n'$ is even and when $n'$ is odd. If $n'$ is even, $j$ is odd and if $n'$ is odd, $j$ is even. 
	The two cases are very similar with only slight changes in a few details. Here, we will do the case when $n'$ is even. 
	The sum above can be rewritten as a sum over odd $j$. Using Equations \eqref{eq:L} and \eqref{eq:bigL}, we can expand 
	this formula. 
	\[
		\sum_{\substack{0\le j \leq n' \vspace{.05cm}  \\ j \text{ odd}}} {{2j}\choose{k}} 
		\sum_{\substack{0\le i_t\leq \lambda'_t \\ i_1+\cdots +i_k = j}}\frac{1}{n'}
		\sum_{\ell \mid \gcd(d',i_1, \cdots, i_k)} \mu(\ell) 
		\frac{(n'/\ell)!}{(\frac{\lambda'_1-i_1}{\ell})!(\frac{i_1}{\ell})! \cdots (\frac{\lambda'_k-i_k}{\ell})!(\frac{i_k}{\ell})!}.
	\]
	Since $\ell \mid\gcd(d',i_1,i_2, \cdots, i_k)$, certainly $\ell\mid d'$. Also, for any fixed $\ell$ and for any choice 
	of $i_1, \ldots, i_k$, we have that $\ell\mid i_t$ for $1\leq t \leq k$ and thus $\ell \mid j$. Since $j$ is always odd, 
	we must also always have that $\ell$ is odd. We can rewrite the above formula, now moving the sum over odd 
	$\ell\mid d$ to the front. We make the following substitutions for the indices in our equation, letting $i$, $r$, $a_t$, 
	and $\gamma_t$ be such that $j = \ell i$, $n' = \ell r$, $i_t = \ell a_t$, $\lambda'_t = \ell\gamma_t$ for all 
	$1\leq t\leq k$. After these substitutions, our resulting formula is the following:
	\[
		\sum_{\substack{\ell\mid d' \\ \ell \text{ odd}}} 
		\sum_{\substack{ 0\le i \leq r \vspace{.05cm} \\ i \text{ odd}}} {{2\ell i}\choose{k}} \mu(\ell)  \frac{1}{n'} 
		\sum_{\substack{0\le a_t\leq \gamma_t \\ a_1+\cdots + a_k = i}}
		\frac{(r)!}{ (\gamma_1-a_1)!(a_1)!\cdots (\gamma_k-a_k)!(a_k)!}
	\]
	We simplify this formula by noticing first that the right-most sum is equal to 
	$\frac{(r)!}{\gamma_1!\cdots \gamma_k!} {{r}\choose{i}}$ by Lemma \ref{lem:long}. We obtain the following formula 
	for the right hand sum of Equation \eqref{eq}:
	\begin{equation}
	\label{eq:rhs}
		\frac{1}{n'} \sum_{\substack{\ell\mid d' \\ \ell \text{ odd}}}  \mu(\ell) \frac{(r)!}{\gamma_1!\cdots \gamma_k!} 
		\sum_{\substack{0\le  i \leq r \\ i \text{ odd}}} {{2\ell i}\choose{k}}  {{r}\choose{i}}.
	\end{equation}
	Finally, we combine the two summations from Equation \eqref{eq} which we have found to be equal to 
	Equations \eqref{eq:lhs} and \eqref{eq:rhs}. After combining like terms, we obtain the formula:
	\begin{align*}
		\frac{1}{n'} \sum_{\substack{\ell\mid d'\\ \ell \text{ odd}}}  \mu(\ell) \frac{r!}{\gamma_1!\cdots \gamma_k!} \cdot 
		& \frac{1}{2}\left[ 2\sum_{\substack{i \text{ odd}\\ i \leq r}} {{2\ell i}\choose{k}}  {{r}\choose{i}}- 
		\sum_{i = 1}^{r} {{2\ell i}\choose{k}} {{r}\choose{i}}\right] \\ 
		& =  \sum_{\substack{\ell\mid d'\\ \ell \text{ odd}}}  \mu(\ell)  \frac{r!}{\gamma_1!\cdots \gamma_k!} 
		 \cdot  \frac{1}{2\ell r}  \left[ - \sum_{i = 1}^{r} (-1)^i {{2\ell i}\choose{k}} {{r}\choose{i}}\right] .
	\end{align*}
	As before, by Lemma \ref{lem:bincoef}, the right most term vanishes except when $r \leq k$. Since 
	$n'/\ell =r\leq k$, we have $n' \leq \ell k$. But $\ell \mid d'$ and $d'\mid \lambda'_t$ for all $1\leq t \leq k$ where 
	$\sum_t\lambda'_t = n'$. Therefore, we must have that $\ell = d' = \lambda'_t$ for all $1\leq t \leq k$. It follows that 
	$\lambda' = (d'^k)$ and therefore we must have that $\lambda = (d^k)$. It follows that if $\lambda \neq (d^k)$, then 
	the sum $\sum_{m=0}^{n-k}(-1)^m \al{\lambda}{m} = 0$.  
	
	Suppose now that $\lambda = (d^k)$. Notice that $d' = \ell$ is odd. 
	For $d'$ odd, we thus obtain:
	\[
		- \frac{1}{2kd'}  \mu(d') k! \cdot \sum_{i = 1}^{k} (-1)^i {{2d'i}\choose{k}} {{k}\choose{i}} =  
		\frac{1}{kd}  \mu(d) k! \cdot \sum_{i = 1}^{k} (-1)^i {{di}\choose{k}} {{k}\choose{i}}. 
	\]
	Recall we are dealing with the case when $n'$ is even and thus $d'$ is odd. Therefore, we must have $k$ even. 
	Therefore, by Lemma \ref{lem:incex}, the rightmost sum is $d^k$. Equation \eqref{eq:deven} follows. 
\end{proof}

\section{Distribution of the descent set of $\C_n$} \label{sec:cons}

In this section, we prove some consequences of Theorem \ref{thm:main} using representation theory. We will let 
$\mathcal{U}(\lambda)$ denote the set of $\lambda$-unimodal permutations.

Recall the definitions of $\rho$ and $\chi$ from Section \ref{sec:chi}. We will show that Theorem \ref{thm:main} implies that the descent sets of elements of $\C_n$ are equi-distributed 
with the descent sets of the standard Young tableaux that form a basis to the representation $\rho$. That is to say, 
for any given $D\subseteq[n-1]$, 
\[
	|\{\pi \in \C_n : \Des(\pi) = D\}| = |\{T \in \mathcal{B}_\rho : \Des(T) = D\}|
\]
where $\mathcal{B}_\rho$ is the basis of representation $\rho$ and the descent set of a given standard Young tableau 
$T$ is defined to be $\Des(T) = \{1\leq i \leq n-1 : i+1 \text{ lies strictly south of } i\}$. 

To prove this, we must first introduce a few definitions from \cite{char}. We say that a subset $D\subset[n-1]$ is 
\emph{$\lambda$-unimodal} if $D\setminus S(\lambda)$ is the disjoint union of intervals of the form 
$[s_{t-1}(\lambda) + \ell_t, s_{t}(\lambda)]$ where $1\leq \ell_t \leq \lambda_t$  for all $1\leq t\leq k$. Note that a 
permutation $\pi\in\S_n$ is $\lambda$-unimodal if and only if its descent set $D$ is $\lambda$-unimodal. 

Consider a given set of combinatorial objects, $\mathcal{B}$, and descent map 
$\Des: \mathcal{B} \to \mathcal{P}([n-1])$ which sends each element $b\in \mathcal{B}$ to a subset 
$\Des(B) \subseteq [n-1]$. If $\rho'$ is some complex representation of $\S_n$, the we call $\mathcal{B}$ a 
\emph{fine set} for $\rho'$ if the character of $\rho'$ satisfies: 
\[
	\chi^{\rho'}_\lambda = \sum_{b\in \mathcal{B}^\lambda} (-1)^{|\Des(b)\setminus S(\lambda)|} 
\]
where $\mathcal{B}^\lambda$ are the elements of $\mathcal{B}$ whose descent set is $\lambda$-unimodal. For 
example, Theorem \ref{thm:main} proves that $\C_n$ is a fine set for the representation $\rho$ defined above.

The following two propositions from \cite{char} will be useful.

\begin{proposition}[{\cite[Cor. 6.7]{char}}]
\label{prop:cor6.7}
	If sets $\mathcal{B}_1$ and $\mathcal{B}_2$ are both fine sets for the same representation, then their descent 
	sets are equi-distributed.
\end{proposition}

\begin{proposition}[{\cite[Thm 2.1]{char}}]
\label{prop:thm2.1}
	Any Knuth class $\C$ of shape $\nu$ is a fine set for the irreducible representation $\S^\nu$ of $\S_n$.
\end{proposition}

Here, a Knuth class is the set of permutations which result in the same insertion tableau when performing the 
Robinson--Schensted--Knuth (RSK) algorithm. For reference, see \cite{Stanley2}.

Denote by $B_\rho$ the set (or possibly, multiset) of standard Young tableaux which form a basis to the representation $\rho$. Recall that any representation of $\S_n$ can be written as the direct sum of some irreducible representations $\S^\nu$ (possibly with multiplicity); that is, $$\rho = \bigoplus_{\nu \in V} \S^\nu.$$ Then the basis of $\rho$ is then defined to be all tableaux $B_\rho = \{T \in \SYT(\nu) : \nu \in V\}$ with multiplicities.

\begin{proposition}\label{prop:basis is fine}
$B_\rho$ is a fine set. 
\end{proposition}
\begin{proof}
Suppose $\chi^\nu$ is the $\S_n$-character of the irreducible representation $\S^\nu$.
By Proposition \ref{prop:thm2.1}, we have that
$$
\chi^\nu_\lambda = \sum_{\pi \in \C\cap \mathcal{U}(\lambda)} (-1)^{|\Des(\pi)\setminus S(\lambda)|}
$$
where $\C$ is any Knuth class of shape $\nu$. In performing RSK, the descent set of the permutations is the same as the descent set of the recording tableau $Q$ (see \cite{schu}). Therefore, we can rewrite this sum over $\lambda$-unimodal permutations in a given Knuth class as a sum over all tableaux $Q$ of shape $\nu$ whose descent set is $\lambda$-unimodal.
$$
\chi^\nu_\lambda = \sum_{Q \in \SYT(\nu)\cap \mathcal{T}(\lambda)} (-1)^{|\Des(Q)\setminus S(\lambda)|}
$$
where $\mathcal{T}(\lambda)$ is the set of tableaux whose descent set is $\lambda$-unimodal  and $\SYT(\nu)$ is the set of standard Young tableaux of shape $\nu$.

Finally, $\rho$ can be written as the direct sum of some irreducible representations $\S^\nu$ and the character of the representation is the sum of the irreducible characters for the representations that appear in this direct sum. Therefore, the theorem follows.
\end{proof}

\begin{theorem}\label{thm:equi} 
The descent sets of elements of $\C_n$ are equi-distributed with the descent sets of the standard Young tableaux  in $B_\rho$.
\end{theorem}\begin{proof}
	By Theorem \ref{thm:main}, $\C_n$ is a fine set for the representation $\rho$. By Proposition 
	\ref{prop:basis is fine}, the set of standard Young tableaux that form a basis of the representation $\rho$ is 
	also a fine set for $\rho$. Therefore, the theorem follows from Proposition \ref{prop:cor6.7}. 
\end{proof}

It should be noted that this statement of equi-distribution of descent sets in Theorem \ref{thm:equi} is a special case 
of Theorem 2.2 in \cite{note}, which itself is a reformulation of Theorem 2.1 in \cite{gessel}. 
 
As another consequence of Theorem \ref{thm:main}, we recover Theorem \ref{thm:sn-1cn}, a result of Elizalde 
\cite{sergi} stating that the number of permutations of $\S_{n-1}$ with descent set $D$ is equal to the number of 
permutations of $\C_{n}$ whose descent set is either $D$ or $D\cup[n-1]$. In \cite{sergi}, this is proved directly with a 
complicated bijection. In the proof of Theorem \ref{thm:sn-1cn}, we show that the result also follows from Theorem 
\ref{thm:main} and Propositions \ref{prop:regular} and \ref{prop:sn is fine}. 

For the next proposition, we consider the embedding of $\S_{n-1}$ into $\S_n$ by the following injective mapping. For some $\pi\in \S_{n-1}$, we obtain a permutation $\pi'\in \S_n$ by letting $\pi'_i = \pi_i$ for all $1\leq i\leq n-1$ and $\pi'_n = n$. 

\begin{proposition}
\label{prop:regular}
	The restriction of $\rho$ to $\S_{n-1}$ is isomorphic to the regular representation.
\end{proposition}
\begin{proof}
	The primitive linear representation $\rho$ of $H$ acts on $\CC v$ by $(12\ldots n) \cdot v \mapsto e^{2\pi i/n} v$. 
	Since $H$ is cyclic, this determines the action. The induced representation $\rho'$ on $\S_n$ is then an action on 
	$\CC\{\sigma_i v\}$ where the $\sigma_i$ are representatives for the distinct cosets of $H \leq \S_n$. For a 
	given element $\pi \in \S_n$ and coset representative $\sigma_i$, there must be some $\tau \in H$ and $j$ so that 
	$\pi\sigma_i = \sigma_j \tau$. The action of the representation is that $\pi \in \S_n$ acts on basis element 
	$\sigma_i v$ by 
	\[
		\pi \cdot \sigma_i v = \sigma_j \rho(\tau) v.
	\]

	We can take the coset representatives $\sigma_i$ to be the elements of $\S_n$ which fix $n$. There are exactly 
	$(n-1)!$ such permutations and $\{\sigma_i H\}$ are all distinct cosets, which follows from the fact that for $i \neq j$, 
	$\sigma_i^{-1}\sigma_j$ fixes $n$ and therefore can only be in $H$ if  $\sigma_i^{-1}\sigma_j$ is the identity. 
	Notice that these coset representatives form a subgroup isomorphic to $\S_{n-1}$. 

	If we take the restriction of $\rho'$ to $\S_{n-1}$, we act on $\CC\{\sigma_i v\}$ by the elements of $\S_{n-1}$, 
	which are exactly the coset representatives. We have that $\sigma_j \cdot \sigma_i v = \sigma_k v$ for some $k$, 
	since $\tau$ is the identity and thus $\rho(\tau) = 1$. If we set $v = 1$, this action is exactly the action by left 
	multiplication on $\CC\S_{n-1}$, which is the regular representation. 
\end{proof}

Though we present a short proof of the following proposition, we note that it also follows from results in \cite{char}.
\begin{proposition}
\label{prop:sn is fine}
	$\S_n$ is a fine set for the regular representation on $\S_n$.
\end{proposition}
\begin{proof}
	It is well-known (for example, see \cite{Sagan}) that the character of the regular representation $\chi^R$ of 
	$\S_n$ takes values 
	\[
		\chi^R(\pi) = \begin{cases} n! & \pi =1 \\ 0 & \text{otherwise.}\end{cases}
	\]
	Therefore, from the definition of a fine set, it suffices to show that
	\begin{equation}
	\label{eqn:reg} 
		\sum_{\pi \in \mathcal{U}(\lambda)} (-1)^{|\Des(\pi) \setminus S(\lambda)|}  = 
		\begin{cases} n! & \lambda = (1^n) \\ 0 & \text{otherwise.}\end{cases}
	\end{equation}
	If $\lambda = 1^n$, then every permutation is $\lambda$-unimodal and each one contributes 1 to the sum since 
	$|\Des(\pi) \setminus S(\lambda)| = |\Des(\pi) \setminus [n-1]| = 0$. Therefore, when $\lambda = (1^n)$, the sum 
	is indeed $n!$.
	
	Now, suppose $\lambda \neq (1^n)$. Consider the following map $\varphi: \mathcal{U}(\lambda) \to \mathcal{U}(\lambda)$. 
	Take $i\geq1$ to be the smallest positive integer such that $\lambda_i>1$. Then $\pi_i\pi_{i+1}\ldots\pi_{i-1+\lambda_i}$ is 
	a unimodal segment of length $\lambda_i>1$. By switching the positions of the largest and second largest elements of this 
	segment, we obtain a unimodal segment with either one more or one less descent. Let $\varphi(\pi)$ be the permutation you 
	obtain by making this change. Then $\varphi$ is an involution on $\mathcal{U}(\lambda)$ which changes 
	$|\Des(\pi) \setminus S(\lambda)|$ by one. Therefore, there must be a bijection between $\lambda$-unimodal permutations 
	with $|\Des(\pi) \setminus S(\lambda)|$ odd and those with $|\Des(\pi) \setminus S(\lambda)|$ even. It follows that the sum 
	must be 0 in this case.
\end{proof}

\begin{theorem}[{\cite{sergi}}]
\label{thm:sn-1cn}
	The number of permutations in $\S_{n-1}$ with descent set $D\subseteq[n-2]$ is equal to the number of permutations in 
	$\C_{n}$ whose descent set is either $D$ or $D\cup[n-1]$.
\end{theorem}
\begin{proof}
	By Proposition \ref{prop:sn is fine}, we know that $\S_{n-1}$ is a fine set for the regular representation of $\S_{n-1}$.  
	Denote the restriction of the character $\chi$ to $\S_{n-1}$ by $\chi^r$. Consider the injection $\iota: \S_{n-1} \to \S_n$ 
	sending $\tau \in \S_{n-1}$ to $\pi\in\S_n$ defined by $\pi(i) = \tau(i)$ for $1\leq i\leq n-1$ and $\pi(n) = n$. This allows us 
	to think of $\S_{n-1}$ as a subgroup of $\S_n$. We have $\chi^r(\tau) = \chi(\tau)$ when $\tau \in \S_{n-1} \leq \S_n$, 
	and $\chi^r(\tau) = 0$ when $\tau \in\S_n\setminus \S_{n-1}$. Recall that by Theorem \ref{thm:main}, 
	\[
		\chi_\lambda = \sum_{\pi \in \C_\lambda} (-1)^{|\Des(\pi) \setminus S(\lambda)|}.
	\]
	Suppose $\lambda^r$ is a composition of $n-1$ and $\lambda = (\lambda_1^r, \lambda_2^r, \ldots, \lambda_{k-1}^r, 1)$. 
	Additionally, for $\pi \in \S_n$, let $\Des^r(\pi)$ denote the \emph{$r$-descent set} defined to be $\Des(\pi)\setminus\{n-1\}$. 
	Then any element of $\C_\lambda$ is $\lambda^r$-unimodal with respect to $\Des^r$, that is to say, the $r$-descent set of 
	an element of $\C_\lambda$ is $\lambda^r$-unimodal as a set. It follows that
	\[
		\chi_{\lambda^r}^r = \sum_{\pi \in \C_\lambda} (-1)^{|\Des^r(\pi) \setminus S(\lambda^r)|}.
	\]
	Therefore, it follows that $\C_n$ (equipped with the $r$-descent set) is a fine set for the restricted representation. 
	By Proposition \ref{prop:cor6.7}, it follows that for a given $D\subseteq[n-2]$, the number of permutations in 
	$\S_{n-1}$ with descent set $D$ is equal to the number of permutations in $\C_n$ with $r$-descents set $D$, 
	from which the theorem follows.
\end{proof}

%

\subsection*{Acknowledgements} 

The author would like to thank Yuval Roichman, Sergi Elizalde, and Zajj Daugherty for helpful discussions. The author would also like to thank the referee for helpful comments and suggestions. 

\bibliographystyle{amsplain}
\bibliography{lambdabib}

\end{document}